\documentclass{amsart} %[11pt, a4paper, twoside, onecolumn]

\usepackage{amsmath,amssymb,amsthm,amsfonts,amsxtra,mathtools}
\usepackage{enumerate,verbatim,mathrsfs,comment,color,url}
\usepackage[all,2cell,ps]{xy}
\usepackage[pagebackref]{hyperref}
\usepackage{graphicx}
\usepackage{tikz-cd}
\usepackage{verbatim}
\usepackage{stackengine}
\usepackage{cleveref}

\DeclareMathOperator{\ann}{ann}
\DeclareMathOperator{\Ass}{Ass}

\DeclareMathOperator{\Ext}{Ext}

\DeclareMathOperator{\Hom}{Hom}

\DeclareMathOperator{\Image}{Im}
\DeclareMathOperator{\indeg}{indeg}

\DeclareMathOperator{\Ker}{Ker}

\DeclareMathOperator{\Tor}{Tor}

\renewcommand{\ge}{\geqslant}
\renewcommand{\le}{\leqslant}

\newcommand{\bn}{\mathbb{N}}
\newcommand{\bz}{\mathbb{Z}}

\newcommand{\fa}{\mathfrak{a}}
\newcommand{\fm}{\mathfrak{m}}

\newcommand{\fp}{\mathfrak{p}}
\newcommand{\fq}{\mathfrak{q}}
\newcommand{\fu}{\mathfrak{u}}
\renewcommand{\iff}{if and only if }
\newcommand{\lra}{\longrightarrow}
\newcommand{\xlra}{\xlongrightarrow}

\theoremstyle{plain}
\newtheorem{theorem}{Theorem}[section]
\newtheorem{lemma}[theorem]{Lemma}
\newtheorem{proposition}[theorem]{Proposition}

\theoremstyle{definition}
\newtheorem{definition}[theorem]{Definition}

\newtheorem{example}[theorem]{Example}

\newtheorem{notations}[theorem]{Notations}

\newtheorem{para}[theorem]{}

\newtheorem{setup}[theorem]{Setup}

\theoremstyle{remark}
\newtheorem{remark}[theorem]{Remark}

\numberwithin{equation}{section}

\title[Asymptotic v-numbers of graded (co)homology modules]{Asymptotic v-numbers of graded (co)homology modules involving powers of an ideal}
\author[D.~Ghosh]{Dipankar Ghosh}
\address{Department of Mathematics, Indian Institute of Technology Kharagpur, West Bengal - 721302, India}
\email{dipankar@maths.iitkgp.ac.in, dipug23@gmail.com}
\urladdr{\url{https://orcid.org/0000-0002-3773-4003}}

\author[S.~Pramanik]{Siddhartha Pramanik}
\address{Department of Mathematics, Indian Institute of Technology Kharagpur, West Bengal - 721302, India}
\email{siddharthap@kgpian.iitkgp.ac.in, pramaniksiddhartha2@gmail.com}

%; Secondary 13D07, 13B22

\begin{document}
    \pagenumbering{arabic}
    \thispagestyle{empty}
    \begin{abstract}
    Let $R$ be a Noetherian $\bn$-graded ring. Let $L$, $M$ and $N$ be finitely generated graded $R$-modules with $N \subseteq M$. For a homogeneous ideal $I$, and for each fixed $k \in \bn$, we show the asymptotic linearity of v-numbers of the graded modules $ \Ext_R^{k}(L,{I^{n}M}/{I^{n}N})$ and $\Tor_k^{R}(L,{I^{n}M}/{I^{n}N})$ as functions of $n$. Moreover, under some conditions on $\Ext_R^k(L,M)$ and $\Tor_k^R(L,M)$ respectively, we prove similar behaviour for v-numbers of $\Ext_R^{k}(L,{M}/{I^{n}N})$ and $ \Tor_k^{R}(L,{M}/{I^{n}N})$. The last result is obtained by proving the asymptotic linearity of v-number of $(U+I^{n}V)/I^{n}W$, where $U$, $V$ and $W$ are graded submodules of a finitely generated graded $R$-module such that $W \subseteq V$ and $(0:_{U}I) = 0$.
    \end{abstract}
    \maketitle
    
\section{Introduction}
    The notion of Vasconcelos invariant of a homogeneous ideal in a polynomial ring over a field was introduced in \cite{CSTPV} to analyze the asymptotic behaviour of the minimum distance of projective Reed-Muller type codes. In the literature, this numerical invariant is known as v-number, and has been named after Wolmer Vasconcelos. The notion has been easily extended from homogeneous ideals to graded modules in \cite[Def.~1.4]{FG24}. The asymptotic behaviour of Vasconcelos invariants of powers of homogeneous ideals, or graded modules involving powers of ideals, or graded filtrations, is the current interest of contemporary researchers, see, for instances, \cite{Co24}, \cite{FS}, \cite{FG24}, \cite{Fi24}, \cite{BMK}, \cite{FS24}, \cite{VS24}, \cite{KNS24} and \cite{FM24}. In this article, we investigate the asymptotic behaviour of Vasconcelos invariants of certain graded (co)homology (namely, Ext and Tor) modules involving powers of a homogeneous ideal.

    \begin{setup}\label{setup}
        Unless specified, let $R = R_{0}[x_1,\dots,x_d]$ be a Noetherian $\bn$-graded ring. Let $L$ and $M$ be finitely generated $\bz$-graded $R$-modules, and $I$ be a homogeneous ideal of $R$. Let $N$ be a graded submodule of $M$. Let $J$ be a reduction ideal of $I$, generated by homogeneous elements $y_1,\dots, y_c$ of degree $d_1\le \cdots \le d_c$ respectively.
    \end{setup}

    The set of all associated prime ideals of the $R$-module $M$ is denoted by $\Ass_{R}(M)$. For $n \in \bz$, $M_{n}$ denotes the $n$th graded component of $M$.

    \begin{definition}\label{defn:v-num-M}
        For $\fp \in \Ass_{R}(M)$, the local v-number (or Vasconcelos invariant) of $M$ at $\fp$ is defined to be 
        $$v_{\fp}(M) := \inf \{n : \text{there exists } x \in M_{n} \text{ such that } \fp = (0 :_{R} x)  \},$$
        where $(0 :_{R} x)$ denotes the set of elements of $R$ which annihilate $x$. The v-number (or Vasconcelos invariant) of $M$ is defined as 
        $$ v(M) := \inf \{ v_{\fp}(M) : \fp \in \Ass_{R}(M) \} .$$
        By convention, $v(0) = \infty$.
    \end{definition}

    In \cite{Br79}, Brodmann showed that the set $\Ass_{R}({M}/{I^{n}M})$ stabilizes for all $n \gg 0$. Motivated by Brodmann's result, in \cite{Co24}, Conca proved that when $R$ is a domain, the function $v(R/I^{n})$ is eventually linear in $n$, i.e., $v(R/I^{n}) = an +b$ for all $n \gg 0$, where $a$ and $b$ are some constants. Ficarra-Sgroi showed this result independently, when $R$ is a polynomial ring over a field, see \cite[Thm.~3.1]{FS}. Fiorindo-Ghosh strengthen these results in \cite[Thm.~2.14]{FG24} by proving that $v({M}/{I^{n}M})$ is eventually linear in $n$ provided $(0:_MI)=0$. Note that \cite[Thm.~2.14]{FG24} deals with the asymptotic behaviour of $v({M}/{I^{n}N})$ under some conditions on $M$ and $N$.
    % It is known that the sets $\Ass_{R}({I^{n}M}/{I^{n}N})$ and $\Ass_{R}({M}/{I^{n}N})$ also stabilize for all $n \gg 0$, due to McAdam-Eakin \cite{ME79} and Katz-West \cite{KW04} respectively. In this regard, Fiorindo-Ghosh \cite{FG24} proved that the functions $v({I^{n}M}/{I^{n}N})$ and $v({M}/{I^{n}N})$ both are eventually linear in $n$.

In classical multiplicity theory, it is known that the length function $\lambda(M/I^nM)$ is given by a polynomial (called Hilbert-Samuel polynomial) in $n$ for all $n\gg 0$, where $M$ is a finitely generated module over a Noetherian local ring $R$, and $I$ is an ideal of $R$ such that $\lambda(M/IM)$ is finite. In this context, Kodiyalam in \cite[Thm.~2]{Ko93} showed that for each fixed $k\ge 0$, the functions $\lambda(\Ext_R^k(L,M/I^nM))$ and $\lambda(\Tor_k^R(L,M/I^nM))$ are eventually polynomials in $n$, provided $\lambda(L\otimes M)$ is finite, where $L$ is also a finitely generated module over $R$.
    
The module theoretic definition of Vasconcelos invariant has opened the path to examine the invariant in a more general way. For each $n\in\mathbb{N}$, there are natural isomorphisms $ \Ext_{R}^0(R,M/I^{n}N) \cong M/I^{n}N \cong \Tor_0^{R}(R, M/I^{n}N)$. Moreover, for each $k \in \bn$, the $R$-modules $\Ext_{R}^k(L,M/I^{n}N)$ and $\Tor_k^{R}(L, M/I^{n}N)$ have natural $\bz$-graded structures. Fix $k \in \bn$. In \cite[Cor.~3.5]{KW04}, Katz-West showed that for all $n\gg 0$, the sets $\Ass_{R}(\Ext_R^{k}(L,{M}/{I^{n}N}))$ and $\Ass_{R}(\Tor_{k}^{R}(L,{M}/{I^{n}N}))$ are independent of $n$. See \cite{GP19} for a concise proof of this result. The stability of $\Ass_{R}(\Tor_{k}^{R}(L,R/I^{n}))$ for $n\gg 0$ was first proved in \cite[Thm.~1]{MS93} by Melkersson-Schenzel. Keeping these results in mind, a natural question arises whether the Vasconcelos invariants of $\Ext_R^{k}(L,{M}/{I^{n}N})$ and $\Tor_{k}^{R}(L,{M}/{I^{n}N})$ as functions of $n$ are linear. The aim of this article is to address this question. In this regard, we mainly prove the following theorem. Here $(0:_{L}I)$, as defined in \ref{not:colon}, is isomorphic to $\Hom_R(R/I,L)$. Thus $(0:_{L}I)=0$ \iff $I$ contains an $L$-regular element (possibly, $L=0$).

% \old{
% \begin{theorem}[See Theorems~\ref{thm:main-Ext} and \ref{thm:main-Tor}]\label{thm:main}
%     With {\rm \Cref{setup}}, fix $k \in \bn$.
%     \begin{enumerate}[\rm (1)]
%     \item 
%     Let $\big(0:_{\Ext_{R}^{k}(L,M)}I\big) = 0$.
%     % $($e.g., $\Ext_{R}^{k}(L,M)=0$$)$. 
%     Consider $\fp \in \Ass_{R}\big(\Ext_R^{k}(L,{M}/{I^{n}N})\big)$ for all $ n \gg 0$. Suppose that $I \subseteq \fp$. Then, there exist $a_{\fp}\in \{ d_1,\ldots,d_c\}$ and $b_{\fp} \in \bz$ such that
%     $$v_{\fp}\big(\Ext_R^{k}(L,{M}/{I^{n}N})\big) = a_{\fp}n+b_{\fp} \; \text{ for all } n \gg 0.$$
%     \item 
%     Let $\big(0:_{\Tor_k^R(L,M)}I\big) = 0$. Consider $\fp \in \Ass_{R}\big(\Tor_k^R(L,{M}/{I^{n}N})\big)$ for all $n \gg 0$. Suppose that $I \subseteq \fp$. Then, there exist $a_{\fp}\in \{ d_1,\ldots,d_c\}$ and $b_{\fp} \in \bz$ such that
%     $$v_{\fp}\big(\Tor_k^R(L,{M}/{I^{n}N})\big) = a_{\fp}n+b_{\fp} \; \text{ for all } n \gg 0.$$
%     \end{enumerate}
% \end{theorem}
% }

\begin{theorem}[See Theorems~\ref{thm:main-Ext} and \ref{thm:main-Tor} for stronger results]\label{thm:main-simple}
    With {\rm \Cref{setup}}, fix $k \in \bn$. Set $H := \Ext_R^k(L,M)$ and $H_n := \Ext_R^k(L,M/I^n M)$, or $H := \Tor_k^R(L,M)$ and $H_n := \Tor_k^R(L,M/ I^n M)$ for all $n\ge 0$. Assume that $(0 :_{H} I) = 0$. Then $H_n$ is eventually zero, or $v(H_n)$ is eventually a linear function of $n$ whose leading coefficient is the degree of a minimal generator of $I$.
\end{theorem}

\begin{remark}
    It is to be noted that we actually prove more stronger results than what is stated in \Cref{thm:main-simple}. Indeed, with {\rm \Cref{setup}}, fix $k \in \bn$. Set $H := \Ext_R^k(L,M)$ and $H_n := \Ext_R^k(L,M/I^n N)$, or $H := \Tor_k^R(L,M)$ and $H_n := \Tor_k^R(L,M/ I^n N)$ for all $n\ge 0$. Suppose $ I^{n_0}M \subseteq N$ for some $n_0\ge 0$ (e.g., $N=M$, or $N=\fa M$ for some homogeneous ideal $\fa$ with $I \subseteq \sqrt{\fa}$). Let $\fp\in\Ass_R(H_n)$ for all $n\gg 0$. Then the local v-number $v_{\fp}(H_n)$ is eventually linear in $n$, see Theorems~\ref{thm:main-Ext} and \ref{thm:main-Tor}.
\end{remark}

This article is arranged as follows. In \Cref{sec2}, we discuss the asymptotic linearity of v-numbers of $\Ext_R^k(L,I^nM/I^nN)$ and $\Tor_k^R(L,I^nM/I^nN)$ as functions of $n$ for each fixed $k\ge 0$. In \Cref{sec3}, first we show that each of $\Ext_R^k(L,M/I^nN)$ and $\Tor_k^R(L,M/I^nN)$ can be expressed as graded module of the form $(U+I^{n}V)/I^{n}W$ for all $n\gg 0$. Then, in \Cref{thm:UVW}, we prove the asymptotic linearity of the local v-number of $(U+I^{n}V)/I^{n}W$ provided $(0:_{U}I) = 0$. We conclude \Cref{sec3} by proving Theorems~\ref{thm:main-Ext} and \ref{thm:main-Tor}. Finally, in \Cref{sec:exam}, we provide a number of examples which complement our results. Most notably, Examples~\ref{example1} and \ref{example3} ensure that the conditions $(0:_{U}I) = 0$, $\big(0:_{\Ext_{R}^{k}(L,M)}I\big) = 0$ and $\big(0:_{\Tor_k^R(L,M)}I\big) = 0$ in Theorems~\ref{thm:UVW}, \ref{thm:main-Ext} and \ref{thm:main-Tor} respectively cannot be removed.

% \noindent {\it Notations.}
\begin{notations}\label{not:colon}
With \Cref{setup}, we use the following notations:
\begin{center}
$(0:_{L}I):= \{ x\in L : Ix=0 \}, \quad \Gamma_{I}(L) = \bigcup_{n \ge 1}\left(0:_{L} I^{n}\right), \quad \ann_{L}(I) := (0:_{L}I)$
\end{center}
and $\indeg(L) := \inf\{ n : L_n \neq 0 \} $. By convention, $\indeg(0) = \infty$.
\end{notations}

\section{Asymptotic v-numbers of $\Ext_R^k(L,I^nM/I^nN)$ and $\Tor_k^R(L,I^nM/I^nN)$}\label{sec2}

In this section, we analyze the asymptotic behaviour of Vasconcelos invariants of the modules $\Ext_R^k(L,I^nM/I^nN)$ and $\Tor_k^R(L,I^nM/I^nN)$ as functions of $n$ for every fixed $k\ge 0$.
    
% \old{\begin{para}\label{para:stability}
%     With \Cref{setup}, let $\mathscr{R}(I) = \bigoplus_{n\ge 0}I^{n}$ be the Rees algebra of $I$. Then the Rees module $\mathscr{R}(I,M) = \bigoplus_{n\ge 0}I^{n}M$ is a finitely generated graded $\mathscr{R}(I)$-module. Clearly, $\mathscr{R}(I,N) = \bigoplus_{n\ge 0}I^{n}N$ is a graded $\mathscr{R}(I)$-submodule of $\mathscr{R}(I,M)$. So the quotient $\frac{\mathscr{R}(I,M)}{\mathscr{R}(I,N)} = \bigoplus_{n \ge 0} \frac{I^{n}M}{I^{n}N}$ is a finitely generated graded $\mathscr{R}(I)$-module.
%     % Thus by \cite[Remark 2.4 (1)]{Pu13}, 
%     It follows that for each $k \in \bn$, the modules $\Ext_{R}^{k}\Big(L,\frac{\mathscr{R}(I,M)}{\mathscr{R}(I,N)}\Big) = \bigoplus_{n \ge 0}\Ext_R^{k}\left(L,\frac{I^{n}M}{I^{n}N}\right)$ and $\Tor_{k}^{R}\left(L,\frac{\mathscr{R}(I,M)}{\mathscr{R}(I,N)}\right) = \bigoplus_{n \ge 0}\Tor_{k}^{R}\left(L,\frac{I^{n}M}{I^{n}N}\right)$ are finitely generated $\bn$-graded over $\mathscr{R}(I)$. Hence, by \cite[Lem.~2.1]{MS93}, % \cite[Thm.~3.4]{We04}, 
%     for each $k\in\bn$, the sets $ \Ass_{R}(\Ext_R^{k}(L,{I^{n}M}/{I^{n}N})) $ and $\Ass_R(\Tor_{k}^{R}(L,{I^{n}M}/{I^{n}N}))$ are independent of $n$ for all $n \gg 0$.
% \end{para}}

\begin{para}
    With \Cref{setup}, let $\mathscr{R}(I) = \bigoplus_{n\ge 0}I^{n}$ be the Rees algebra of $I$. Then the Rees module $\mathscr{R}(I,M) = \bigoplus_{n\ge 0}I^{n}M$ is a finitely generated graded $\mathscr{R}(I)$-module. Clearly, $\mathscr{R}(I,N) = \bigoplus_{n\ge 0}I^{n}N$ is a graded $\mathscr{R}(I)$-submodule of $\mathscr{R}(I,M)$. So the quotient $\frac{\mathscr{R}(I,M)}{\mathscr{R}(I,N)} = \bigoplus_{n \ge 0} \frac{I^{n}M}{I^{n}N}$ is a finitely generated graded $\mathscr{R}(I)$-module.
\end{para}

%%%%%    \begin{lemma}
%        With setup \ref{setup}, for each $i \ge 0$ the module $\bigoplus_{n \ge 0}\Tor_{i}^{R}(L,\frac{I^{n}M}{I^{n}N})$ is finite $R[J]$-module.
%%%%%    \end{lemma}
%%%    \begin{proof}
%        The quotient $\frac{R[I,M]}{R[I,N]} = \bigoplus_{n \ge 0} \frac{I^{n}M}{I^{n}N}$ is a finite $R[I]$-module. Hence 
%        $$Tor_{i}^{R}(L,\frac{R[I,M]}{R[I,N]}) = \bigoplus_{n \ge 0}\Tor_{i}^{R}(L,\frac{I^{n}M}{I^{n}N})$$
%        is a finite $R[I]$-module. Since $J$ is a reduction of $I$, $R[I]$ is finite $R[J]$-module, so the result follows.
%%%%%%    \end{proof}

\begin{para}\label{para:stability}
    With \Cref{setup}, let $\mathcal{N}:=\bigoplus_{n \ge 0} N_n$ be a finitely generated graded module over the graded ring $\mathscr{R}(I)$. Let $\mathbb{F}: \cdots \to F_n \to F_{n-1} \to \cdots \to F_1 \to F_0 \to 0$ be a graded free resolution of $L$ consisting of finitely generated graded free $R$-modules. Then, for $a \in I^r$, the map $ N_n \xrightarrow{a} N_{n+r}$ induces a map $\Hom_R(\mathbb{F},N_n) \xlongrightarrow{a} \Hom_R(\mathbb{F},N_{n+r})$. It induces a map on the cohomology modules $\Ext_R^k(L,N_n) \xlongrightarrow{a} \Ext_R^k(L,N_{n+r})$. Thus $\bigoplus_{n \ge 0}{\Ext_R^k(L,N_n)}$ is a graded module over the graded ring $\mathscr{R}(I)$. Consequently,
    \begin{align*}
    \Ext_R^k \big(L, \mathcal{N} \big) = \bigoplus_{n \ge 0}{\Ext_R^k(L,N_n)}.
%    & = \bigoplus_{n \ge 0} H^k(\Hom_R(\mathbb{F}, N_n)) \\
%    & = H^k \big(\bigoplus_{n \ge 0} \Hom_R(\mathbb{F}, N_n)\big) \\
%    & = H^k \big(\Hom_R(\mathbb{F}, \bigoplus_{n \ge 0} N_n) \big) \\ 
%    & = \Ext_R^k \big(L, \bigoplus_{n \ge 0} N_n \big),
    \end{align*}
    is a finitely generated graded $\mathscr{R}(I)$-module. As $J$ is a reduction ideal of $I$, it follows that $\Ext_R^k \big(L, \mathcal{N} \big)$ is a finitely generated graded $\mathscr{R}(J)$-module. Similarly, $\bigoplus_{n \ge 0} \Tor_k^R(L,N_n)$ is a finitely generated graded $\mathscr{R}(J)$-module.
    Hence, by \cite[Lem.~2.1]{MS93},
    %\cite[Thm.~3.4]{We04}, 
    for each $k\in\bn$, the sets $ \Ass_{R}\big(\Ext_R^{k}(L, N_n)\big) $ and $\Ass_R \big(\Tor_{k}^{R}(L,N_n) \big)$ are independent of $n$ for $n \gg 0$. In particular, $\mathcal{N}$ can be taken as $\mathscr{R}(I,M)/\mathscr{R}(I,N)$.
\end{para}

\begin{para}[Bigraded structures on Ext and Tor]\label{para:bigrading}
    With \Cref{setup}, let $\deg(x_i)= f_i$ for $1 \le i \le d$. Then $\mathscr{R}(J) = \bigoplus_{n \ge 0} J^{n}$ can be considered as an $\bn^2$-graded ring, where the $(n,l)$th graded component of $\mathscr{R}(J)$ is the $l$th graded component of $J^{n}$ for every $(n,l)\in \bn^2$. Thus $\mathscr{R}(J) = R_{0}[x_1,\dots,x_d,y_1,\dots,y_c]$, where $\deg(x_i)=(0,f_i)$ for $1 \le i \le d$, and $\deg(y_j) = (1,d_j)$ for $1 \le j \le c$. Set $I^nM := 0$ for all $n<0$. Hence, setting
    % \begin{center}
    $\Ext_{R}^{k}\big(L, \mathscr{R}(I,M)/\mathscr{R}(I,N)\big)_{(n,l)} := \Ext_R^{k} \big(L, I^{n}M/I^{n}N\big)_{l}$ for all $(n,l)\in \bz^2$,
    % \end{center}
    the module $\Ext_{R}^{k}\big(L, \mathscr{R}(I,M)/\mathscr{R}(I,N)\big)$ becomes a $\bz^2$-graded $\mathscr{R}(J)$-module. Since we are only changing the grading, this bigraded $\mathscr{R}(J)$-module is still finitely generated. Thus both $\Ext_{R}^{k}\big(L, \mathscr{R}(I,M)/\mathscr{R}(I,N)\big)$ and $\Tor_{k}^{R}\big(L, \mathscr{R}(I,M)/\mathscr{R}(I,N)\big)$ are finitely generated $\bz^2$-graded $\mathscr{R}(J)$-modules.
\end{para}

    % Now we can prove our result concerning the ??. We shall prove the results for the Ext modules only, as the proof for Tor modules follows exactly the same direction.

    % \begin{para}\label{para:delta}
    %     With \Cref{setup}, set $\delta := \inf \left\{ j : y_{j} \notin \sqrt{ \ann_{\mathscr{R}(J)} \left(\Ext_R^{k}\left(L,\frac{\mathscr{R}(I,M)}{\mathscr{R}(I,N)}\right)\right) } \right \}$
    % \end{para}

    The following result shows that for each fixed $k\ge 0$, both initial degree and Vasconcelos invariant of $\Ext_R^{k}(L,{I^nM}/{I^{n}N})$ as functions of $n$ are eventually linear functions, which share the same slope.

\begin{proposition}\label{prop:v-Ext-lin}
    With {\rm \Cref{setup}}, fix $k \in \bn$. Set $H_n := \Ext_R^{k}(L, I^{n}M/I^{n}N)$ for all $n \ge 0$. Let $\fp \in \Ass_{R}(H_n)$ for all $n \gg 0$. Set $\mathcal{H} := \Ext_{R}^{k}\big(L, \mathscr{R}(I,M)/\mathscr{R}(I,N)\big)$ and $\delta := \inf \Big\{ j : 1\le j\le c, \; y_{j} \notin \sqrt{ \ann_{\mathscr{R}(J)} (\mathcal{H}) } \Big\}$. Then, there exist $a \in \{d_{\delta},\dots, d_c \}$ and $b \in \bz$ such that
    \begin{center}
       $ v_{\fp}( H_n ) = an+b$ for all $n \gg 0$.
    \end{center}
    Moreover, both the functions $\indeg(H_n)$ and $v(H_n)$ are eventually linear in $n$ with the same leading coefficient $d_{\delta} \in \{ d_1,\dots,d_c\}.$
\end{proposition}
    
% \old{\begin{proposition}%\label{prop:v-Ext-lin}
%     Fix $k \in \bn$. With {\rm \Cref{setup}} let $\fp \in \Ass_{R}\big(\Ext_R^{k}(L, I^{n}M/I^{n}N)\big)$ for all $n \gg 0$. Then, there exist $a \in \{d_{\delta},\dots, d_c \}$ and $b \in \bz$ such that
%     \begin{center}
%        $ v_{\fp}\big(\Ext_R^{k}(L, I^{n}M/I^{n}N)\big) = an+b$ for all $n \gg 0$,
%     \end{center}
%     where $\delta := \inf \Big\{ j : y_{j} \notin \sqrt{ \ann_{\mathscr{R}(J)} \big(\Ext_{R}^{k}\big(L, \mathscr{R}(I,M)/\mathscr{R}(I,N)\big)\big) } \Big\}$. Moreover, both the functions $\indeg(\Ext_R^{k}(L,{I^{n}M}/{I^{n}N}))$ and $v(\Ext_R^{k}(L,{I^{n}M}/{I^{n}N}))$ are eventually linear in $n$ with the same leading coefficient $d_{\delta} \in \{ d_1,\dots,d_c\}.$
% \end{proposition}}
    
\begin{proof}
    In view of \ref{para:bigrading}, $\mathcal{H}$ is a finitely generated $\bz^2$-graded module over the $\bn^2$-graded ring $\mathscr{R}(J)$, where the bigrading of $\mathscr{R}(J)$ is shown in \ref{para:bigrading}. Note that $ \mathcal{H}_{(n,*)} := \bigoplus_{l\in\mathbb{Z}}  \mathcal{H}_{(n,l)}$ is same as $H_n$ for each $n \ge 0$. Hence the proposition follows from \cite[Thm.~2.8]{FG24}.
\end{proof}

In a similar way, setting $\mathcal{H}:= \Tor_{k}^{R}\big(L, \mathscr{R}(I,M)/\mathscr{R}(I,N)\big)$, one obtains the counterpart of \Cref{prop:v-Ext-lin} for Tor modules.

\begin{proposition}\label{prop:v-Tor-lin}
    With {\rm \Cref{setup}}, fix $k \in \bn$. Set $H_n := \Tor_{k}^{R}(L, I^{n}M/I^{n}N)$ for all $n \ge 0$. Let $\fp \in \Ass_{R}(H_n)$ for all $n \gg 0$. Set $\mathcal{H} := \Tor_{k}^{R}\big(L, \mathscr{R}(I,M)/\mathscr{R}(I,N)\big)$ and $\delta := \inf \Big\{ j : 1 \le j \le c, \; y_{j} \notin \sqrt{ \ann_{\mathscr{R}(J)} (\mathcal{H}) } \Big\}$. Then, there exist $a \in \{d_{\delta},\dots, d_c \}$ and $b \in \bz$ such that
    \begin{center}
       $ v_{\fp}(H_n) = an+b$ for all $n \gg 0$.
    \end{center}
    Moreover, both the functions $\indeg(H_n)$ and $v(H_n)$ are eventually linear in $n$ with the same leading coefficient $d_{\delta} \in \{ d_1,\dots,d_c\}.$
\end{proposition}

% \old{\begin{proposition}%\label{prop:v-Tor-lin}
%     Fix $k \in \bn$. With {\rm \Cref{setup}} let $\fp \in \Ass_{R}\big(\Tor_{k}^{R}(L, I^{n}M/I^{n}N)\big)$ for all $n \gg 0$. Then, there exist $a \in \{d_{\delta},\dots, d_c \}$ and $b \in \bz$ such that
%     \begin{center}
%        $ v_{\fp}\big(\Tor_{k}^{R}(L, I^{n}M/I^{n}N)\big) = an+b$ for all $n \gg 0$,
%     \end{center}
%     where $\delta := \inf \Big\{ j : y_{j} \notin \sqrt{ \ann_{\mathscr{R}(J)} \big(\Tor_{k}^{R}\big(L, \mathscr{R}(I,M)/\mathscr{R}(I,N)\big)\big) } \Big\}$. Moreover, both the functions $\indeg(\Tor_{k}^{R}(L,{I^{n}M}/{I^{n}N}))$ and $v(\Tor_{k}^{R}(L,{I^{n}M}/{I^{n}N}))$ are eventually linear in $n$ with the same leading coefficient $d_{\delta} \in \{ d_1,\dots,d_c\}.$
% \end{proposition}}

\begin{remark}
    Note that the idea used in this section to prove the asymptotic linearity cannot be applied to prove our main results Theorems~\ref{thm:main-Ext} and \ref{thm:main-Tor} because the modules $\bigoplus_{n \ge 0}\Ext^k_R(L,M/I^n N)$ and $\bigoplus_{n \ge 0}\Tor^R_k(L, M/I^n N)$ both are not necessarily finitely generated as $\bz^2$-graded modules over the ring $\mathscr{R}(I)$.
\end{remark}

\section{Asymptotic v-numbers of graded (co)homology modules}\label{sec3}
% \section{Asymptotic linearity of $v(\Ext_R^k(L,M/I^nN))$ and $v(\Tor_k^R(L,M/I^nN))$}\label{sec3}

% All our results in this section regarding Ext modules holds for Tor modules also and the proofs follow exactly the same way. Hence we shall proof only for Ext modules.
% Here, for every fixed $k\ge 0$, we mainly prove the asymptotic linearity of the v-functions $v(\Ext_R^k(L,M/I^nN))$ and $v(\Tor_k^R(L,M/I^nN))$ under some conditions. 
In order to prove our main results, we need a number of lemmas.
% , which are also used in the next section to show. 
Our first lemma in this regard is observed from the proof of \cite[Prop.~3.4]{KW04}. We also add its proof as the construction of $U$ in this lemma is used to show our main results.
% for the benefit of the reader,

\begin{lemma}\label{lem:homology}
    With $R$ and $I$ as in {\rm \Cref{setup}}, let $A\xrightarrow{\phi} B \xrightarrow{\psi} C$ be a complex of finitely generated graded $R$-modules. Suppose $A'\subseteq A$, $B'\subseteq B$ and $C' \subseteq C$ are graded submodules, which satisfy $\phi(A') \subseteq B'$ and $ \psi(B') \subseteq C'$. For each $n \in \bn$, one has the induced complex:
    \begin{align}\label{eq1.4}
       A/I^{n}A'\xrightarrow{\phi(n)} B/I^{n}B' \xrightarrow{\psi(n)} C/I^{n}C'.
    \end{align} 
    Let $H(n)$ denote the homology of \eqref{eq1.4}. Then, there exists $n_{0} \in \bn $ such that
    \begin{equation*}
        H(n) \cong (U+I^{n-n_{0}}V)/I^{n-n_{0}}W \; \text{ for all } n\ge n_{0}
    \end{equation*}
    and for some graded submodules $U$, $V$ and $W$ of a finitely generated graded $R$-module $Z$ satisfying $W \subseteq V$ and $U = \Ker(\psi)/\Image(\phi)$.
\end{lemma}

\begin{proof}
    % We denote the kernel and image of a homomorphism by $\Ker$ and $\Image$ respectively.
    Note that $C'$ and $\Image(\psi)$ (i.e., image of $\psi$) both are submodules of $C$. So, by the Artin-Rees lemma (cf.~\cite[17.1.6]{SH06}), there exists $n_{0} \in \bn$ such that
    $$ I^{n}C' \cap \Image(\psi) = I^{n-n_{0}}(I^{n_{0}}C' \cap \Image(\psi)) \; \mbox{ for all } n \ge n_{0}.$$
    Then, for each $n \ge n_{0}$, it follows that $\psi^{-1}(I^nC')\subseteq \Ker(\psi) + I^{n-n_{0}}\psi^{-1}(I^{n_{0}}C')$, where $\Ker(\psi)$ denotes the kernel of $\psi$. Consequently, for all $n \ge n_{0}$, $\psi^{-1}(I^nC') = \Ker(\psi) + I^{n-n_{0}}\psi^{-1}(I^{n_{0}}C')$.
    % Since $ \psi(B') \subseteq C'$, one has that $I^nB'\subseteq \psi^{-1}(I^nC')$ for all $n$.
    This yields that
    \begin{align}\label{ker-psi}
        \Ker(\psi(n)) &= \psi^{-1}(I^nC')/I^nB' \\
        &= \big(\Ker(\psi) + I^{n-n_{0}}\psi^{-1}(I^{n_{0}}C')\big)/I^{n}B' \; \mbox{ for all } n \ge n_{0}.\nonumber
    \end{align}
    On the other hand,
    \begin{equation}\label{im-phi}
         \Image(\phi(n))  = (\Image(\phi) + I^{n}B')/I^{n}B' \; \mbox{ for all } n.
    \end{equation}
    Set $V' := \psi^{-1}(I^{n_{0}}C')$ and $ W' := I^{n_{0}}B'$. Then $W'\subseteq V'$ as $ \psi(B') \subseteq C'$. Thus, for each $n \ge n_{0}$, the homology of \eqref{eq1.4} is given by
    \begin{align*} 
        H(n) & = \frac{\Ker(\psi(n))}{\Image(\phi(n))} \cong \frac{\Ker(\psi) + I^{n-n_{0}}V'}{\Image(\phi)+ I^{n-n_{0}}W'} \quad\mbox{[by \eqref{ker-psi} and \eqref{im-phi}]}\\
        & \cong \frac{\Ker(\psi)/\Image(\phi)+\big(\Image(\phi) + I^{n-n_0}V'\big)/\Image(\phi)}{\big(\Image(\phi) + I^{n-n_0}W'\big)/\Image(\phi)} = \frac{U+I^{n-n_{0}}V}{I^{n-n_{0}}W},
    \end{align*}
    where
    % \begin{center}
        $U := \Ker(\psi)/\Image(\phi)$, $V := (V' +\Image(\phi))/\Image(\phi)$ and $W := (W' +\Image(\phi))/ \Image(\phi)$.
    % \end{center}
    Set $Z := B/\Image(\phi)$. Clearly, $Z$ is a finitely generated graded $R$-module that contains $U$, $V$ and $W$ as graded submodules. Note that $W\subseteq V$ as $W'\subseteq V'$.
\end{proof}

\begin{remark}
    It can be observed that \Cref{lem:homology} also holds true in non-graded setup as the same proof works in general setup.
\end{remark}

As consequences of \Cref{lem:homology}, $\Ext_{R}^{k}(L, M/I^{n}N)$ and $\Tor_k^R(L, M/I^{n}N)$ can be expressed as quotients of graded modules involving powers of $I$ as shown below.

\begin{lemma}\label{lem:ext:interpret}
    With {\rm \Cref{setup}}, fix $k \in \bn$. Then, there exists $q\in \bn$ such that
    % and finitely generated graded $R$-modules $U,V,W$ and $Z$ such that $U \subseteq Z$ and $W \subseteq V \subseteq Z$ such that
    $$ \Ext_{R}^{k}(L, M/I^{n}N) \cong (U+I^{n-q}V)/I^{n-q}W \; \text{ for all } n \ge q$$
    and for some graded submodules $U$, $V$ and $W$ of a finitely generated graded $R$-module $Z$ satisfying $W \subseteq V$ and $U\cong \Ext_R^k(L,M)$.
\end{lemma}

\begin{proof}
    Let $(\alpha) : F_{k+1} \to F_{k} \to F_{k-1}$ be the part of a graded free resolution of $L$ over $R$ consisting of finitely generated free $R$-modules. Dualizing $(\alpha)$ with respect to $M$, one obtains a complex of the form $A\to B \to C$, where each of $A$, $B$ and $C$ is isomorphic to a direct sum of some finite copies of $M$. Moreover, the homology of this complex is nothing but $\Ext_R^k(L,M)$. Now, for each $n \ge 0$, dualizing $(\alpha)$ with respect to $M/I^nN$, one gets a complex of the form $A/I^nA' \to B/I^nB' \to C/I^nC'$ for some graded submodules $A' \subseteq A$, $B' \subseteq B $ and $C' \subseteq C$. Its homology is $\Ext_{R}^{k}(L, {M}/{I^{n}N})$. Hence the desired result follows from \Cref{lem:homology}.
    % \begin{equation}\label{complex-1}
    %     \Hom_{R}(F_{k-1},M/I^{n}N) \lra \Hom_{R}(F_{k},M/I^{n}N)\lra \Hom_{R}(F_{k+1},M/I^{n}N).
    % \end{equation}
    % where  The complex \eqref{complex-1} can be expressed as
    % \begin{align} \label{eq:1.3}
    %     \frac{A}{I^{n}A'}\lra \frac{B}{I^{n}B'}\lra \frac{C}{I^{n}C'}
    % \end{align}
    % for some finitely generated graded $R$-modules $A' \subseteq A,B' \subseteq B $ and $C' \subseteq C$. Thus, by \Cref{lem:homology}, for all $n \ge q$ the homology module of the complex \ref{eq:1.3} is isomorphic to $\frac{U+I^{n-q}V}{I^{n-q}W} $.
\end{proof}

Tensoring $(\alpha)$ with $M$ and $M/I^nN$ over $R$ respectively (instead of dualizing), by a similar argument as in the proof of \Cref{lem:ext:interpret}, one obtains the following.

\begin{lemma}\label{lem:tor:interpret}
    With {\rm \Cref{setup}}, fix $k \in \bn$. Then, there exists $q\in \bn$ such that
    $$ \Tor_k^R(L, M/I^{n}N) \cong (U+I^{n-q}V)/I^{n-q}W \; \text{ for all } n \ge q$$
    and for some graded submodules $U$, $V$ and $W$ of a finitely generated graded $R$-module $Z$ satisfying $W \subseteq V$ and $U\cong \Tor_k^R(L,M)$.
\end{lemma}
In view of Lemmas~\ref{lem:ext:interpret} and \ref{lem:tor:interpret}, we need to analyze the behaviour of the local v-number of $(U+I^nV)/I^nW$ as a function of $n$. So, in \Cref{thm:UVW}, we study the asymptotic behaviour of this function under the condition that $ ( 0 :_{U} I ) = 0 $. Our next two lemmas and \Cref{thm:UVW} are shown with the following setup.

\begin{setup}\label{setup-2}
    Let $R$, $I$ and $J$ be as in \Cref{setup}. Suppose $U$, $V$ and $W$ are graded submodules of a finitely generated $\bz$-graded $R$-module $Z$ such that $W \subseteq V$.
\end{setup}

% With \Cref{setup-2}, we first analyze the asymptotic behaviour of the function $v_{\fp}\left(\frac{U+I^{n}V}{I^{n}W}\right)$, where $\fp\in\Ass_{R}\left(\frac{U+I^{n}V}{I^{n}W}\right)$ for all $n\gg 0$, see \Cref{thm:UVW}. \new{Then, using \Cref{thm:UVW} we give our main result \Cref{thm:main-Ext}.}
% In order to analyze the asymptotic behaviour of local v-numbers of $\frac{U+I^{n}V}{I^{n}W}$, we use the following lemma.
% In order to prove \Cref{thm:UVW}, we first prepare a couple of lemmas.

\begin{lemma}\label{lem:ann}
    With {\rm \Cref{setup-2}}, denote $M_n := (U+V)/I^nW $ and $N_n := (U+I^nV)/I^nW$ for each $ n \ge 0$. Assume that $(0:_{U+V} I) = 0$. Suppose $\fa$ and $\fu$ be homogeneous ideals of $R$ such that $I \subseteq \fu$. Set $M'_n := \ann_{M_n}(\fu)$ and $N'_n := \ann_{N_n}(\fu)$. Then, for all $n \gg 0$,
    $$ \frac{M'_n}{M'_n \bigcap \Gamma_{\fa} \left(M_n \right) } = \frac{N'_n}{N'_n \bigcap \Gamma_{\fa} \left(N_n\right) }.$$
\end{lemma}

% \old{\begin{lemma}\label{lem:ann}
%     With {\rm \Cref{setup-2}}, let $(0:_{U+V} I) = 0$. Suppose $\fa$ and $\fu$ be homogeneous ideals of $R$ such that $I \subseteq \fu$. Then, for all $n \gg 0$,
%     $$ \frac{\ann_{\frac{U+V}{I^{n}W}}(\fu)}{\ann_{\frac{U+V}{I^{n}W}}(\fu) \bigcap \Gamma_{\fa} \left(\frac{U+V}{I^{n}W}\right) } = \frac{\ann_{\frac{U+I^{n}V}{I^{n}W}}(\fu)}{\ann_{\frac{U+I^{n}V}{I^{n}W}}(\fu) \bigcap \Gamma_{\fa} \left(\frac{U+I^{n}V}{I^{n}W}\right) }.$$
% \end{lemma}}
    
\begin{proof}
    Observe that 
    \begin{align*}
        \left(I^{n+1}W :_{(U+I^{n+1}V)} \fu\right) 
        &\subseteq \left(I^{n+1}W :_{(U+V)} \fu\right) 
            \subseteq \left(I^{n+1}(U+V) :_{(U+V)} \fu\right) \\ 
        & \subseteq  \left(I^{n+1}(U+V) :_{(U+V)} I\right) = I^{n}(U+V) \quad  \text{for all } n \gg 0, 
    \end{align*} 
    where the last equality follows from \cite[Lem.~(4)]{Br79} as $(0:_{U+V} I) = 0$. So
    \begin{align*}
    \left(I^{n+1}W :_{(U+I^{n+1}V)} \fu\right) &= \left(I^{n+1}W :_{(U+I^{n+1}V)} \fu\right) \cap I^{n}(U+V) \\
    &=  \left(I^{n+1}W :_{I^{n}(U+V)} \fu\right) \quad  \text{for all } n \gg 0.
    \end{align*}
    Going modulo $I^{n+1}W$ both sides, as graded submodules of $(U+V)/I^{n+1}W$, one obtains that $\ann_{\frac{U+I^{n+1}V}{I^{n+1}W}}(\fu) = \ann_{\frac{I^{n}(U+V)}{I^{n+1}W}}(\fu)$ for all $n \gg 0$. On the other hand, in view of \cite[Equ.~(2.6)]{FG24}, one has that $\ann_{\frac{U+V}{I^{n+1}W}}(\fu) = \ann_{\frac{I^{n}(U+V)}{I^{n+1}W}}(\fu)$ for all $n \gg 0$. Combining both the equalities, it follows that
    % \new{
    % \begin{align}%\label{equ:numerator}
    %     \ann_{M_{n+1}}(\fu) = \ann_{N_{n+1}}(\fu) \text{ for all } n \gg 0.
    % \end{align}}
    \begin{align}\label{equ:numerator}
        \ann_{\frac{U+V}{I^{n+1}W}}(\fu) = \ann_{\frac{U+I^{n+1}V}{I^{n+1}W}}(\fu) \text{ for all } n \gg 0.
    \end{align} 
    As 
    %\new{$ N_{n+1} \bigcap \Gamma_{\fa} \left(M_{n+1}\right) = \Gamma_{\fa} \left(N_{n+1}\right) $} 
    $ \frac{U+I^{n+1}V}{I^{n+1}W} \bigcap \Gamma_{\fa} \left(\frac{U+V}{I^{n+1}W}\right) = \Gamma_{\fa} \left(\frac{U+I^{n+1}V}{I^{n+1}W}\right) $, the equalities in \eqref{equ:numerator} yield that
    % \new{\begin{align}%\label{equ:denominator}
    %     \ann_{M_{n+1}}(\fu) \bigcap \Gamma_{\fa} \left(M_{n+1}\right)
    %     & = \ann_{N_{n+1}}(\fu) \bigcap \Gamma_{\fa}\left(M_{n+1}\right) \\
    %     & = \ann_{N_{n+1}}(\fu) \bigcap \Gamma_{\fa}\left(N_{n+1}\right) \nonumber
    % \end{align}}
    \begin{align}\label{equ:denominator}
        \ann_{\frac{U+V}{I^{n+1}W}}(\fu) \bigcap \Gamma_{\fa} \left(\frac{U+V}{I^{n+1}W}\right)
        & = \ann_{\frac{U+I^{n+1}V}{I^{n+1}W}}(\fu) \bigcap \Gamma_{\fa}\left(\frac{U+V}{I^{n+1}W}\right) \\
        & = \ann_{\frac{U+I^{n+1}V}{I^{n+1}W}}(\fu) \bigcap \Gamma_{\fa}\left(\frac{U+I^{n+1}V}{I^{n+1}W}\right) \nonumber
    \end{align}
    for all $ n \gg 0$. Hence, the desired equalities follow from \eqref{equ:numerator} and \eqref{equ:denominator}.  
\end{proof}

The following is a consequence of Artin-Rees lemma.

\begin{lemma}\label{lem:colon}
    With {\rm \Cref{setup-2}}, there exists an integer $q \ge 0$ such that 
    $$ (0:_{U+I^{n}V} I) = (0:_{U} I) \text{ for all } n \ge q.$$
\end{lemma}

\begin{proof}
    Set $T := (0:_{Z} I)$. Then, by Artin-Rees lemma (cf.~\cite[17.1.6]{SH06}), there exists $n_{0} \in \bn $ such that $I^{n}V \cap (U+T) = I^{n-n_{0}}\big(I^{n_0}V \cap (U+T)\big)$ for all $n \ge n_{0}$. Let $q = n_{0}+1$. Consider $n \ge q$, and an element $x \in (0:_{U+I^{n}V} I) $. Then $x\in T$, and $x=u+v$ for some $u \in U$ and $v \in I^{n}V$. Hence $v=-u+x \in I^{n}V \cap (U+T) =  I^{n-n_0}\big(I^{n_0}V \cap (U+T)\big) \subseteq I^{n-n_0}(U+T) = I^{n-n_0}U $ as $I T = 0$ and $n\ge q > n_0$. Thus $v\in U$. It follows that $x=u+v \in T \cap U = (0:_{Z} I)\cap U = (0:_{U} I) $. So $(0:_{U+I^{n}V} I) \subseteq (0:_{U} I)$ for all $n \ge q$. Since $(0:_{U} I) \subseteq (0:_{U+I^{n}V} I)$ for all $n \in \bn$, the desired equalities follow.
\end{proof}

It is known that the set $\Ass_{R}\big((U+I^nV)/I^nW\big)$ stabilizes for all $n \gg 0$, cf.~\cite[Prop.~3.4 and its proof]{KW04}. We prove that the local v-number of $(U+I^nV)/I^nW$ is eventually linear as a function of $n$ provided $(0:_{U}I) = 0$. This result considerably strengthens \cite[Thm.~1.1]{Co24} and \cite[Thm.~1.9.(1)]{FG24}.

\begin{theorem}\label{thm:UVW}
    With {\rm \Cref{setup-2}}, let $(0:_{U}I) = 0$. Let $\fp \in \Ass_{R}\big((U+I^nV)/I^nW\big)$ for all $n \gg 0$. Assume that $I \subseteq \fp$. Then, there exist $a \in \{ d_1,\dots,d_c\}$ and $b \in \bz$ such that $v_{\fp}\big((U+I^nV)/I^nW\big) = an+b$ for all $n \gg 0$.%%% Moreover, if $\Ass_{R}^{\infty}(\frac{U+IV}{IW}) = \Ass_{R}(\frac{I^{n}(U+I^{q}V)}{I^{n+q+1}W})$ for all $n\gg 0$, then $v_{\fp}(\frac{U+I^{n+1}V}{I^{n+1}W}) = v_{\fp}(\frac{I^{n}(U+I^{q}V)}{I^{n+q+1}W})$ for all $n \gg 0$
\end{theorem}

\begin{proof}
    In view of \Cref{lem:colon}, there exists $q \in \bn$ such that $(0:_{U+I^nV} I) = (0:_{U} I)$ for all $n\ge q$. Set $V_1 := I^{q}V$ and $W_1 := I^{q}W$. Then $W_1\subseteq V_1$. Moreover, $(0:_{U+V_1} I) = (0:_{U}I) = 0$. Set $\mathscr{H} := \mathscr{R}(I,U+V_1)/ \mathscr{R}(I,IW_1)$. Then $\mathscr{H}$ is a finitely generated $\bz^2$-graded module over the $\bn^2$-graded ring $\mathscr{R}(J)$, where the $(n,l$)th graded component of $\mathscr{H}$ is given by $\big({I^n(U+V_1)}/{I^{n+1}W_1}\big)_l$ for each $(n,l) \in \bz^2$, and the gradation of $\mathscr{R}(J)$ is as shown in \ref{para:bigrading}. Using the ideal $\fp$, set $X:= \{ \fq \in \Ass_R\big((U+I^{n}V)/I^{n}W\big) \text{ for all } n \gg 0: \fp \subsetneq \fq \}$. Let $\fa = \prod_{\fq \in X}\fq $ when $X \neq \emptyset$, otherwise $\fa = R$. Let $\mathscr{L} = \ann_{\mathscr{H}}(\fp)/\ann_{\mathscr{H}}(\fp) \cap \Gamma_{\fa}(\mathscr{H})$. Then $\mathscr{L}$ is also a finitely generated $\bz^2$-graded $\mathscr{R}(J)$-module, where the bigraded structure of $\mathscr{L}$ is induced by that of $\mathscr{H}$. It follows that    \begin{equation}\label{L-H}
        \mathscr{L}_{(n,*)} = \ann_{\mathscr{H}_{(n,*)}}(\fp)/\ann_{\mathscr{H}_{(n,*)}}(\fp) \cap \Gamma_{\fa}(\mathscr{H}_{(n,*)}) \; \text{ for all } n,
    \end{equation}
    where $\mathscr{H}_{(n,*)} := \bigoplus_{l\in\mathbb{Z}} \mathscr{H}_{(n,l)} = I^{n}(U+V_1)/I^{n+1}W_1$. In view of \cite[Thm.~2.8]{FG24}, there exists $b_1 \in \bz$ such that
    \begin{equation}\label{indeg-L}
        \indeg(\mathscr{L}_{(n,*)}) = d_{\delta}n+b_1 \; \mbox{ for all } n \gg 0,
    \end{equation}
    where $\delta = \inf\left\{ j : y_j \notin \sqrt{\ann_{\mathscr{R}(J)}(\mathscr{L})}, 1\le j \le c \right\}$. For each $n\ge 0$, set
    % To avoid writing big expressions, set
    \begin{equation}\label{D-E}
        \mathscr{D}_n := (U+I^nV_1)/I^nW_1 \quad \mbox{and} \quad \mathscr{E}_n := (U+V_1)/I^nW_1.%\quad \mbox{for all } n\ge 0.
    \end{equation}
    Thus, using the notations described above, for all $n\gg 0$,
    \begin{align*}
        &v_{\fp}\big((U+I^{n+q+1}V)/I^{n+q+1}W\big) = v_{\fp}\big((U+I^{n+1}V_1)/I^{n+1}W_1\big) = v_{\fp}(\mathscr{D}_{n+1}) \\
        &= \indeg\big(\ann_{\mathscr{D}_{n+1}}(\fp)/\ann_{\mathscr{D}_{n+1}}(\fp) \cap \Gamma_{\fa}(\mathscr{D}_{n+1})\big) \qquad \text{[by \cite[Lem.~1.5]{FG24}]} \\
        &=\indeg\big(\ann_{\mathscr{E}_{n+1}}(\fp)/\ann_{\mathscr{E}_{n+1}}(\fp) \cap \Gamma_{\fa}(\mathscr{E}_{n+1})\big) \;\,\qquad \text{[by \Cref{lem:ann}]}  \\
        &=\indeg\big(\ann_{\mathscr{H}_{(n,*)}}(\fp)/\ann_{\mathscr{H}_{(n,*)}}(\fp) \cap \Gamma_{\fa}(\mathscr{H}_{(n,*)})\big) \;\; \text{[by \cite[Lem.~2.13]{FG24}]}\\
        &= \indeg(\mathscr{L}_{(n,*)}) = d_{\delta} n + b_1.\qquad\qquad\qquad\qquad\qquad \text{[by \eqref{L-H} and \eqref{indeg-L}]}
    \end{align*}
    Therefore, setting $a:=d_{\delta}$ and $b := b_1 - (q+1)d_{\delta}$, one gets the desired result.    
    % \old{Now by \cite[Thm.~2.8]{FG24} there exist $b_1 \in \bz$ such that $\indeg(\mathscr{L}_{(n,*)}) = d_{\delta_{\fp}}n+b_1$ for all $n \gg 0$, where 
    % \begin{align*}
    %     \delta_{\fp} = \inf \left\{ j : y_j \notin \sqrt{\ann_{\mathscr{R}(J)}(\mathscr{L})}, 1\le j \le c \right\}
    % \end{align*}
    % Now for $n \gg 0$, we have 
    % \begin{align*}
    %     v_{\fp}\left(\frac{U+I^{n+1}V_1}{I^{n+1}W_1}\right)            &=\indeg\left(\frac{\ann_\frac{{U+I^{n+1}V_1}}{I^{n+1}W_1}(\fp)}{\ann_{\frac{{U+I^{n+1}V_1}}{I^{n+1}W_1}}(\fp) \bigcap \Gamma_{V}\frac{U+I^{n+1}V_1}{I^{n+1}W_1)}}\right) \quad \text{[by \cite[Lem. 1.5]{FG24}]} \\
    %     &=\indeg\left(\frac{\ann_\frac{{U+V_1}}{I^{n+1}W_1}(\fp)}{\ann_{\frac{{U+V_1}}{I^{n+1}W_1}}(\fp) \bigcap \Gamma_{V}\frac{U+V_1}{I^{n+1}W_1)}}\right) \qquad \text{[by \Cref{lem:ann}]}  \\
    %     &=\indeg\left(\frac{\ann_\frac{{I^{n}(U+V_1)}}{I^{n+1}W_1}(\fp)}{\ann_{\frac{{I^{n}(U+V_1)}}{I^{n+1}W_1}}(\fp) \bigcap \Gamma_{V}\frac{I^{n}(U+V_1)}{I^{n+1}W_1)}}\right) \quad \text{[by \cite[Lem. 2.13]{FG24}]}\\
    %     &= \indeg(\mathscr{L}_{(n,*)}) \\
    %     &= d_{\delta_{\fp}}n+b_1
    % \end{align*}
    % Thus $v_{\fp}\left(\frac{U+I^{n+q+1}V}{I^{n+q+1}W}\right) = v_{\fp}\left(\frac{U+I^{n+1}V_1}{I^{n+1}W_1}\right) = d_{\delta_{\fp}}n+b_1$ for all $n \gg 0$.  Setting $b = b_1 - (q+1)d_{\delta_{\fp}}$, one gets $v_{\fp}\left(\frac{U+I^{n}V}{I^{n}W}\right)= d_{\delta_{\fp}}n+b$ for all $n \gg 0$. This proves the desired result.}
\end{proof}

Now we are in a position to prove the main results of this article.
    
\begin{theorem}\label{thm:main-Ext}
    Fix $k \in \bn$. With {\rm \Cref{setup}},  denote $H := \Ext_R^k(L,M)$ and $H_n := \Ext_R^k(L,M/I^n N)$ for each $n \ge 0$. Let $(0 :_H I) = 0$.
    \begin{enumerate}[\rm (1)]
    \item 
    Let $\fp \in \Ass_{R}(H_n)$ for all $ n \gg 0$. Suppose that $I \subseteq \fp$. Then, there exist $a_{\fp}\in \{ d_1,\ldots,d_c\}$ and $b_{\fp} \in \bz$ such that $v_{\fp}(H_n) = a_{\fp}n+b_{\fp} \; \text{ for all } n \gg 0.$    
    \item
    Suppose $ I^{n_0}M \subseteq N$ for some $n_0 \in \bn$ $($e.g., $N=M$, or $N=\fa M$ for some homogeneous ideal $\fa$ with $I \subseteq \sqrt{\fa})$. Then, there exist $a \in \{ d_1,\dots,d_c\}$ and $b \in \bz\cup\{\infty\}$ such that $v(H_n) = an+b \text{ for all } n \gg 0.$
    \end{enumerate}
\end{theorem}

\begin{proof}
    (1) In view of \Cref{lem:ext:interpret}, there exists $q \ge 0$ such that for all $n \ge q$,  $\Ext_R^{k} (L,M/{I^{n}N}) \cong (U+I^{n-q}V)/{I^{n-q}W}$ for some graded submodules $U$, $V$ and $W$ of a finitely generated graded $R$-module $Z$ satisfying $W \subseteq V$ and $U\cong \Ext_R^k(L,M)$. From the given condition on $\Ext_{R}^{k}(L,M)$, one has that $(0:_{U}I) = 0$. So, by \Cref{thm:UVW} there exist $a_{\fp} \in \{d_1,\dots,d_c\} $ and $b_1 \in \bz$ such that $v_{\fp}((U+I^{n}V)/{I^{n}W}) = a_{\fp}n+b_1$ for all $n \gg 0$. Hence, setting $b_{\fp} = b_1 -a_{\fp}q$, it follows that 
    \begin{align*}
        v_{\fp} \big(\Ext_R^{k} (L,{M}/{I^{n}N})\big) 
        & = v_{\fp} \big((U+I^{n-q}V)/{I^{n-q}W}\big) \\
        & = a_{\fp}n +b_{\fp} \quad \mbox{for all } n \gg 0.
    \end{align*}
         
    (2) Set $\mathcal{A} := \Ass_{R}(\Ext_R^{k}(L,{M}/{I^{n}N}))$ for all $n \gg 0$. If $\mathcal{A}$ is an empty set, then $\Ext_R^{k}(L,{M}/{I^{n}N})=0$, and hence $v(\Ext_R^{k}(L,{M}/{I^{n}N})) = \infty$ for all $n \gg 0$. So we may assume that $\mathcal{A}$ is non-empty. Note that $\mathcal{A}$ is a finite set. The condition $I^{n_0}M \subseteq N$ for some $n_0 \in \bn$ ensures that $I \subseteq \fp $ for each $\fp \in \mathcal{A}$.  Therefore, since 
    $$ v\big(\Ext_R^{k}(L,{M}/{I^{n}N})\big) = \inf \big\{ v_{\fp}\big(\Ext_R^{k}(L,{M}/{I^{n}N})\big) : \fp \in \mathcal{A}\big\}, $$ 
    the desired result follows from (1) and \cite[2.9]{FG24}.
\end{proof}

We have a similar result for Tor modules. Since its proof goes exactly in the same lines, we state it without the proof.
    
\begin{theorem}\label{thm:main-Tor}
    Fix $k \in \bn$. With {\rm \Cref{setup}}, denote $H := \Tor_k^R(L,M)$ and $H_n := \Tor_k^R(L,M/ I^n N)$ for each $n \ge 0$. Let $\big(0:_{H}I\big) = 0$.
    \begin{enumerate}[\rm (1)]
    \item 
    Let $\fp \in \Ass_{R}(H_n)$ for all $n \gg 0$. Suppose that $I \subseteq \fp$. Then, there exist $a_{\fp}\in \{ d_1,\ldots,d_c\}$ and $b_{\fp} \in \bz$ such that $v_{\fp}(H_n) = a_{\fp}n+b_{\fp} \; \text{ for all } n \gg 0.$
    \item 
    Suppose $ I^{n_0}M \subseteq N$ for some $n_0 \in \bn$ $($e.g., $N=M$, or $N=\fa M$ for some homogeneous ideal $\fa$ with $I \subseteq \sqrt{\fa})$. Then, there exist $a \in \{ d_1,\ldots,d_c\}$ and $b \in \bz\cup\{\infty\}$ such that $v(H_n) = an+b$ for all $ n \gg 0$.
\end{enumerate}
\end{theorem}

\begin{remark}
    Instead of asymptotic local v-number of $\Ext_R^{k}(L,{M}/{I^{n}N})$, one may want to consider $\Ext_R^{k}({M}/{I^{n}N},L)$. Unfortunately, the sets $\Ass_R (\Ext_R^{k}({M}/{I^{n}N},L) $ do not stabilize in general for $n \gg 0$, see \cite[Sec.~4]{Si00}.
\end{remark}

\section{Some examples}\label{sec:exam}

In this section, we provide a number of examples to complement our results. We start with the following which can be derived from \cite[Ex.~3.1]{FG24}. It ensures that the condition $(0:_{U} I) = 0$ in \Cref{thm:UVW} cannot be removed.

\begin{example}\label{example1}
    Let $R = K[X,Y]$ be a standard graded polynomial ring in two variables $X$ and $Y$ over a field $K$. Consider $ U = V = W := R/(XY^b)$, $I := (X^a)$, $\fp := (X)$ and $\fm := (X,Y)$, where $a$ and $b$ are some positive integers. Then 
     % Consider $ U = V = W := R/(X^{a}Y)$, $I := (Y^{b})$, $\fp := (Y)$ and $\fm := (X,Y)$, where $a$ and $b$ are some positive integers. Then
    % \begin{enumerate}[(1)]
    %     \item 
    %     $(0:_{U} I) = X^aU \neq 0$.
    %     \item $\Ass_{R} \big((U+I^{n}V)/{I^{n}W}\big) = \{ \fp,\fm \}$ if $bn\ge 2$, and $\{\fp\}$ if $bn=1$.
    %     \item $v_{\fp} \big((U+I^{n}V)/{I^{n}W}\big) = a$ whenever $bn \ge 2$.
    %     \item $v_{\fm} \big((U+I^{n}V)/{I^{n}W}\big) = bn+(a-2)$ whenever $bn \ge 2$.
    %     \item $v \big((U+I^{n}V)/{I^{n}W}\big) = a$ whenever $bn \ge 2$.
    %  \end{enumerate}
    \begin{enumerate}[(1)]
        \item 
        $(0:_{U} I) = Y^bU \neq 0$.
        \item $\Ass_{R} \big((U+I^{n}V)/{I^{n}W}\big) = \{ \fp,\fm \}$ if $bn\ge 2$, and $\{\fp\}$ if $bn=1$.
        \item $v_{\fp} \big((U+I^{n}V)/{I^{n}W}\big) = b$ whenever $an \ge 2$.
        \item $v_{\fm} \big((U+I^{n}V)/{I^{n}W}\big) = an+(b-2)$ whenever $bn \ge 2$.
        \item $v \big((U+I^{n}V)/{I^{n}W}\big) = b$ whenever $an \ge 2$.
     \end{enumerate}
\end{example}
    
% \begin{proof}
%     For $n \ge 1$, $I^{n}U = (Y^{bn},X^{a}Y)/(X^{a}Y)$, and hence
%     $$(U+I^{n}V)/{I^{n}W} = U/I^{n}U \cong R/(Y^{bn},X^{a}Y).$$
%     Considering the primary decomposition of $(Y^{bn},X^{a}Y)$, one obtains (2). Note that
%     \begin{center}
%     $\fp = \big((Y^{bn},X^{a}Y):_{R} X^{a}\big)$ and $\fm = \big((Y^{bn},X^{a}Y):_{R} X^{a-1}Y^{bn-1} \big)$ whenever $bn \ge 2$.  
%     \end{center}
%     These equalities do not hold true if $X^{a}$ and $X^{a-1}Y^{bn-1}$ are replaced respectively by any other homogeneous element of smaller degree. Consequently, (3) and (4) follow. The statement (5) can be observed from (2), (3) and (4).
% \end{proof}
% 
% \begin{remark}
%     In \Cref{example1}, $(0:_{U} I) \neq 0$, and $v_{\fp}((U+I^{n}V)/{I^{n}W})$ is eventually a constant function of $n$. It ensures that the condition $(0:_{U} I) = 0$ in \Cref{thm:UVW} cannot be removed.
% \end{remark}

\begin{para}\label{para:observation}
    Let $Z$, $Z_1$ and $Z_2$ be finitely generated graded $R$-modules. Denote $Z(i)$ the graded module $Z$, but its graded components are shifted as follows: $Z(i)_n = Z_{n+i}$ for all $n\in \mathbb{Z}$. Then the following elementary facts can be observed from the definition of v-number.
    \begin{enumerate}[\rm (1)]
    \item $v(Z(i)) = v(Z) - i$ for all $i \in \bz$.
    \item $ v(Z_1 \oplus Z_2) = \inf \{ v(Z_1), v(Z_2)\}$.
    \end{enumerate}
\end{para}

In the example below,
% the conditions $\big(0:_{\Ext_R^{k}( L, M)}I\big) = 0$ and $\big(0:_{\Tor_k^{R}( L, M)}I\big) = 0$ hold, consequently
for each fixed $k\ge 0$, the v-functions $v(\Ext_R^{k}(L, M/I^nN))$ and $v(\Tor_k^{R}(L, M/I^nN))$ are asymptotically either $\infty$ or linear in $n$ as shown in Theorems~\ref{thm:main-Ext}.(2) and \ref{thm:main-Tor}.(2) respectively.

\begin{example}\label{example2}
    Let $R = K[X,Y]/(XY)$, where $K$ is a field, and $\deg(X) = \deg(Y) = 1$. Suppose $x$ and $y$ are the residue classes of $X$ and $Y$ in $R$ respectively. Set $L := R/(x)$, $M = N := R$, $I := (y)$ and $\fm := (x,y)$. Then,
    \begin{enumerate}[(1)]
    \item 
    $\big(0:_{\Ext_R^{k}( L, M)}I\big) = 0$ and $\big(0:_{\Tor_k^{R}( L, M)}I\big) = 0$ for each $k \ge 0$.   
    \item
    For all $n \ge 2$, $\Ext_R^{k}(L, M/I^nN) = 0$ if $k \ge 2$ is even, $\Ass_R\big(\Ext_R^{k}( L, M/I^nN)\big) =\{ \fm \}$ if $k = 0$ or $k \ge 1$ is odd, $\Tor_k^{R}(L, M/I^nN) = 0 $ if $k \ge 1$ is odd, and $\Ass_R\big(\Tor_k^{R}( L, M/I^nN)\big) = \{ \fm \}$ if $ k \ge 0$ is even.    
    \item
    $v\big(\Ext_R^{0}( L, M/I^nN)\big) = n-1$ for all $n \ge 2$, and
    \begin{align*} 
    v\big(\Ext_R^{k}( L, M/I^nN)\big) = 
    \left \{  
    \begin{array}{ll} 
    \infty & \text{if $k\ge 2$ is even and $n \ge 1$} \\ 
    n-k-1 & \text{if $k\ge 1$ is odd and $n \ge 1$}. 
    \end{array}
    \right.
    \end{align*}
    % \item for all $n \ge 2$, $v\big(\Tor_0^{R}( L, M/I^nN)\big) =  $ \com{Something is wrong here. It is violating our main result on Tor} and     
    \item 
    For all $n \ge 1$, $v\big(\Tor_k^{R}( L, M/I^nN)\big) = n+k-1 $ if $k\ge 0$ is even, and it is $\infty$ if $k\ge 1$ is odd.
    % \begin{align*} 
    % v\big(\Tor_k^{R}( L, M/I^nN)\big) = 
    % \left \{  
    % \begin{array}{ll} 
    %  n+k-1 & \text{if $k\ge 0$ is even,}  \\     
    %  \infty & \text{if $k\ge 1$ is odd}. 
    % \end{array}
    % \right.
    % \end{align*}
    \end{enumerate}
\end{example}

\begin{proof}
    Observe that
    \begin{align} \label{eq:reso} 
    \mathbb{F}_{L}: \quad \cdots \xlra{y} R(-3) \xlra{x} R(-2) \xlra{y} R(-1) \xlra{x} R \lra 0 
    \end{align}
    is a minimal graded free resolution of $L$ over $R$. Since 
    \begin{equation}\label{eq}
        \Hom_{R}(R(-k),W) \cong W(+k) \mbox{ and }  R(-k) \otimes_{R}W \cong W(-k)
    \end{equation} 
    for any graded $R$-module $W$ and for all $k \in \bz$, it follows that 
    \begin{align} \Hom_{R}(\mathbb{F}_{L},M): \quad 0\lra M \xlra{x} M(+1) \xlra{y} M(+2) \xlra{x} M(+3) \xlra{y} \cdots \label{Hom-L-reso}\\
     \mbox{and~} \mathbb{F}_{L} \otimes_{R} M: \quad \cdots \xlra{y} M(-3) \xlra{x} M(-2) \xlra{y} M(-1) \xlra{x} M \lra 0.\label{Tensor-L-reso}
     \end{align}
    These sequences are acyclic as $M=R$. Thus
    \begin{align*}
        \Ext_R^k(L,M) = 
        \left\{\begin{array}{ll}
        (y) & \mbox{if~} k=0,  \\
        0 & \mbox{if~} k\ge 1,
        \end{array}
        \right.
        \quad \mbox{and}\quad
        \Tor_k^R(L,M) =
        \left\{
        \begin{array}{ll}
        R/(x) & \mbox{if~} k=0, \\
        0 & \mbox{if~} k\ge 1.
        \end{array}
        \right. 
    \end{align*}
    So (1) follows. Since $M/I^nN = R/(y^n)$ for all $n \ge 1$, replacing $M$ by $R/(y^n)$ in \eqref{Hom-L-reso} and \eqref{Tensor-L-reso}, and then computing (co)homologies, one obtains that
    \begin{align}
        \Ext_R^k(L,M/I^nN) = 
        \left \{
        \begin{array}{ll}
        (y)/(y^n) & \mbox{if~} k=0  \\
        \big((x,y^{n-1})/(x,y^n)\big)(+k) & \mbox{if~} k\ge 1 \mbox{~is odd} \\
        0 & \mbox{if~} k\ge 2 \mbox{~is even}
        \end{array}
        \right. \label{Ext-modules-3-cases}\\
        \mbox{and }\;
        \Tor_k^R(L,M/I^nN) =
        \left \{
        \begin{array}{ll}
        R/(x,y^n) & \mbox{if~} k=0  \\
        0 & \mbox{if~} k\ge 1 \mbox{~is odd} \\
        \big((x,y^{n-1})/(x,y^n)\big)(-k) & \mbox{if~} k\ge 2 \mbox{~is even}
        \end{array}
        \right.\label{Tor-modules-3-cases}
    \end{align}
    for all $n \ge 2$. Note that $(x,y^n)$ annihilates the modules in \eqref{Ext-modules-3-cases} and \eqref{Tor-modules-3-cases}. So the set of associated prime ideals of each of the non-zero modules in \eqref{Ext-modules-3-cases} and \eqref{Tor-modules-3-cases} is $\{\fm\}$. Thus one has (2). Note that
    \begin{align*}
    &(y)/(y^n) \cong Ky \oplus \cdots \oplus Ky^{n-1}, \quad %\cong K(-1) \oplus \cdots \oplus K(-n+1), \\
    R/(x,y^n) \cong K\oplus Ky \oplus \cdots \oplus Ky^{n-1} \\%\cong K \oplus K(-1) \oplus \cdots \oplus K(-n+1) \\
    &\mbox{and } \big((x,y^{n-1})/(x,y^n)\big)(i) \cong Ky^{n-1}(i) \mbox{ for all } i\in \bz.% \cong K(-n+i+1),
    \end{align*}
    So $v((y)/(y^n)) = n-1$ for all $n\ge 2$. Also, for all $n\ge 1$, $v(R/(x,y^n))= n-1$ and
    \begin{align*}
        % &v((y)/(y^n)) = n-1,\quad v(R/(x,y^n)) = n-1 \\ 
        % & \mbox{and} \quad  
        v\big(\big((x,y^{n-1})/(x,y^n)\big)(i)\big) = v(Ky^{n-1}) - i = n-i-1 \mbox{ for all } i\in \bz.
    \end{align*}
     Consequently, (3) and (4) follow. 
\end{proof}

Our next example is to assure about the importance of the colon conditions  $\big(0:_{\Ext_R^{k}( L, M)}I\big) = 0$ and $\big(0:_{\Tor_k^{R}( L, M)}I\big) = 0$ in Theorems~\ref{thm:main-Ext} and \ref{thm:main-Tor} respectively.

\begin{example}\label{example3}
    Let $R$, $L$, $I$ and $\fm$ be as in \Cref{example2}, and $M = N := (y)$. Then,
    \begin{enumerate}[(1)]
    % \item $\big(0:_{\Ext_R^{k}( L, M)}I\big) =
    % \left\{
    % \begin{array}{ll}
    % 0 & \mbox{if $k=0$ or $k\ge 1$ is odd}, \\
    % (y)/(y^2)(+k) & \mbox{if $k\ge 2$ is even}.
    % \end{array}
    % \right.$
    \item 
    % \begin{enumerate}[(a)]
    %     \item $\big(0:_{\Ext_R^{k}( L, M)}I\big) = 0$ if $k=0$ or $k\ge 1$ is odd.
    %     \item $\big(0:_{\Ext_R^{k}( L, M)}I\big) = \big((y)/(y^2)\big)(+k)$ if $k\ge 2$ is even.
    % \end{enumerate}
    $\big(0:_{\Ext_R^{k}( L, M)}I\big) = 0$ if $k=0$ or $k\ge 1$ is odd, and this module is $\big((y)/(y^2)\big)(+k)$ if $k\ge 2$ is even.
    \item For $k \ge 0$, $\big(0:_{\Tor_k^{R}( L, M)}I\big) = 0$ if $k$ is even, and it is $\big((y)/(y^2)\big)(-k)$ if $k$ is odd.
    \item 
    $\Ass_R(\Ext_R^{k}( L, M/I^nN)) = \Ass_R(\Tor_k^{R}( L, M/I^nN)) = \{ \fm \}$ for all $k \ge 0$ and $n \ge 1$.
    \item 
    For all $n \ge 1$, $v\big(\Ext_R^{k}( L, M/I^nN)\big) = 
    \left \{  
    \begin{array}{ll} 
     n-k& \text{if $k=0$ or $k\ge 1$ is odd},  \\ 
     -k+1& \text{if $k\ge 2$ is even}.
    \end{array}
    \right.$
    \item 
    For all $k\ge 0$ and $n \ge 1$,
    $ 
    v\big(\Tor_k^{R}( L, M/I^nN)\big) = 
    \left \{  
    \begin{array}{ll} 
    n+k& \text{if $k$ is even},\\     
    k+1& \text{if $k$ is odd}. 
    \end{array}
    \right.$
    \end{enumerate}
\end{example}

\begin{proof}
    As in Example~\ref{example2}, considering $M=(y)$ in \eqref{Hom-L-reso} and \eqref{Tensor-L-reso}, one gets that
    \begin{align*}
    \Ext_R^k(L,M) & =
    \left \{
    \begin{array}{ll}
    M & \mbox{if~} k=0,  \\
    0 & \mbox{if $k\ge 1$ is odd},\\
    \big((y)/(y^2)\big)(+k) & \mbox{if $k\ge 2$ is even}
    \end{array}
    \right. \\ 
    \mbox{and} \quad 
    \Tor_k^R(L,M) & =
    \left \{
    \begin{array}{ll}
    M & \mbox{if~} k=0,  \\
    \big((y)/(y^2)\big)(-k) & \mbox{if $k\ge 1$ is odd},\\
    0 & \mbox{if $k\ge 2$ is even}.
    \end{array}
    \right.    
    \end{align*}
    Consequently, (1) and (2) can be observed. Since $M/I^nN = (y)/(y^{n+1})$ for all $n \ge 1$, replacing $M$ by $(y)/(y^{n+1})$ in \eqref{Hom-L-reso} and \eqref{Tensor-L-reso}, and then computing (co)homologies,
    \begin{align*}
    \Ext_R^k(L,M/I^nN) & =
    \left \{
    \begin{array}{ll}
    (y)/(y^{n+1}) & \mbox{if~} k=0,  \\
    \big((y^n)/(y^{n+1})\big)(+k) & \mbox{if $k\ge 1$ is odd},\\
    (\big(y)/(y^2)\big)(+k) & \mbox{if $k\ge 2$ is even}
    \end{array}
    \right. \\ 
    \mbox{and} \quad 
     \Tor_k^R(L,M/I^nN) & =
    \left \{
    \begin{array}{ll}
    (y)/(y^{n+1}) & \mbox{if~} k=0,  \\
    \big((y)/(y^2)\big)(-k) & \mbox{if $k\ge 1$ is odd},\\
    \big((y^n)/(y^{n+1})\big)(-k) & \mbox{if $k\ge 2$ is even}.
    \end{array}
    \right.    
    \end{align*}
    By similar arguments as in Example~\ref{example2}, one obtains (3), (4) and (5).
    % \old{\\ Since,
    % \begin{align*}
    %     ((y)/(y^{n+1}))(i) \cong (Ky \oplus \cdots \oplus Ky^{n})(i) & \cong K(i-1) \oplus \cdots \oplus K(i-n) \\
    %     \mbox{and~} ((y^n)/(y^{n+1}))(i) \cong Ky^n(i) & \cong K(i-n),
    % \end{align*}  
    % \com{The last isomorphisms are true as K-vector spaces, not as $R$-modules}
    % by using \ref{para:observation}, we get $v((y)/(y^{n+1})(i) = n-i$ and $v((y^n)/(y^{n+1})(i)) = n-i$, for all $i \in \bz$ and $n \ge 1$. From this, we get (3) and (4).}  
\end{proof}

\begin{remark}
    In Example~\ref{example3}, for each fixed $k\ge 1$,
    % the local v-functions of the Ext and Tor modules coincide with their v-functions, as these modules have only one associated prime ideal, namely $\fm$. The
    both $\big(0:_{\Ext_R^{2k}( L, M)}I\big)$ and $\big(0:_{\Tor_{2k-1}^{R}( L, M)}I\big)$ are nonzero,
    % both the modules
    % \begin{center}
    %     $\big(0:_{\Ext_R^{2k}( L, M)}I\big)$ and $\big(0:_{\Tor_{2k-1}^{R}( L, M)}I\big)$ are nonzero,
    % \end{center}
    and the (local) v-numbers $v_{\fm}\big(\Ext_R^{2k}( L, M/I^nN)\big) $ and $v_{\fm}\big(\Tor_{2k-1}^{R}( L, M/I^nN)\big) $ as functions of $n$ are constants for all $n \ge 1$. It ensures that the conditions $\big(0:_{\Ext_R^k( L, M)}I\big)=0$ and $\big(0:_{\Tor_k^{R}( L, M)}I\big)=0$ in Theorems~\ref{thm:main-Ext} and \ref{thm:main-Tor} respectively cannot be omitted.
\end{remark}

\section*{Acknowledgments}
Siddhartha Pramanik would like to thank the Government of India for the financial support through the Prime Minister Research Fellowship for his PhD. The authors thank the anonymous reviewer for carefully reading the manuscript and providing many helpful suggestions.

\end{document}